\documentclass{amsart}
\usepackage{color}
\usepackage{mathabx}
\usepackage{mathrsfs}
 \newtheorem{thm}{Theorem}[section]
 \newtheorem{cor}[thm]{Corollary}

 \theoremstyle{definition}
 
 \theoremstyle{remark}

 \numberwithin{equation}{section}

\newcommand{\C}{\mathbb{C}}
\newcommand{\R}{\mathbb{R}}
\newcommand{\N}{\mathbb{N}}

\begin{document}

%
%
%
%
%
%
%
%
%

\title[Moment functions ...]{Moment functions and exponential\\ monomials on commutative\\ hypergroups}

\author{\.Zywilla Fechner }
\address{Institute of Mathematics\\
				 {\L}\'{o}d\'{z}  University of Technology\\
				90-924 {\L}\'{o}d\'{z} \\
				ul. W\'{o}l\-cza\'{n}\-ska 215, Poland 
        \\ ORCID 0000-0001-7412-6544}
\email{zfechner@gmail.com }

\vbox{\author{Eszter Gselmann}
\address{University of Debrecen\\
 H-4002 Debrecen\\ P.O.Box: 400, Hungary}
\email{gselmann@science.unideb.hu}
}

\author{L\'{a}szl\'{o} Sz\'{e}kelyhidi}
\address{University of Debrecen\\
 H-4002 Debrecen\\ P.O.Box: 400, Hungary
\\ ORCID: 0000-0001-8078-6426}
\email{szekely@science.unideb.hu, lszekelyhidi@gmail.com}

\thanks{
Project no.~K134191 supporting E.~Gselmann and L.~Sz\'{e}kelyhidi, has been implemented by the support provided from the National Research, Development and Innovation Fund of Hungary, financed under the K\_20 funding scheme.
The research of E.~Gselmann has partially been carried out with the help of the project 2019-2.1.11-T\'{E}T-2019-00049,
which has been implemented by the support provided from the National Research, Development
and Innovation Fund of Hungary. 
}



\date{\today}
\dedicatory{Dedicated to the 80\textsuperscript{th} birthday of Professor Ludwig Reich. }

\begin{abstract}
The purpose of this paper is to prove that 
if on a commutative hypergroup an exponential monomial has the property that the linear subspace of all sine functions in its variety is one dimensional, then this exponential monomial is a linear combination of generalized moment functions.
\end{abstract}

\maketitle

\section{Introduction}

A {\it hypergroup} is a locally compact Hausdorff space $X$ equipped with an involution and a convolution operation defined on the space of all bounded complex regular measures on $X$. For the formal definition, historical background and basic facts about hypergroups we refer to \cite{BlH95}. In this paper $X$ denotes a locally compact hypergroup with identity element $o$, involution $\widecheck{}$, 
and convolution $*$. In fact, the quadruple $(X,o,{\widecheck{}}\,,*)$ is what we should call a hypergroup, but for the sake of simplicity we shall call $X$ a hypergroup. In this paper we shall consider commutative hypergroups only, hence we always suppose that $X$ is a locally compact commutative hypergroup.
\vskip.2cm

Given $x$ in $X$ we denote the point mass with support the singleton $\{x\}$ by $\delta_x$ which is a probability measure on $X$, and so is $\delta_x*\delta_y$ whenever $x,y$ are in $X$. For a continuous function $h\colon X\to \mathbb{C}$ the integral
$$ 
\int_X h(t)d(\delta_x*\delta_y)(t)
$$ 
will be denoted by $h(x*y)$. Clearly, $h(x*y)$ is the mathematical expectation of the random variable $h$ on the probability space $(X,\mathcal B, \delta_x*\delta_y)$, $\mathcal B$ being the $\sigma$-algebra of all Borel subsets of $X$. Given $y$ in $X$ the function $x\mapsto h(x*y)$ is the {\it translate} of $h$ by $y$. A comprehensive monograph on the subject is \cite{Sze12}. 
\vskip.2cm

A set of continuous complex valued functions on $X$ is called {\it translation invariant}, if it contains all translates of its elements. A linear translation invariant subspace of all continuous complex valued functions is called a {\it variety}, if it is closed with respect to uniform convergence on compact sets. The smallest variety containing the given function $h$ is called the {\it variety of $h$}, and is denoted by $\tau(h)$. Clearly, it is the intersection of all varieties including $h$. A continuous complex valued function is called an {\it exponential polynomial}, if its variety is finite dimensional. The simplest nonzero exponential polynomial is the one having one dimensional variety: it consists of all constant multiples of a nonzero continuous function. If we normalize that function by taking $1$ at $o$ then we have the concept of an exponential. Recall that $m$ is an {\it exponential} on $X$ if it is a non-identically zero continuous complex-valued function  satisfying $m(x*y)=m(x) m(y)$ for 
each $x,y$ in $X$ and $s$ is an {\it $m$-sine function} on $X$ if it is a continuous complex-valued function fulfilling $s(x*y)=s(x)m(y)+m(x)s(y)$ for each $x,y$ in $X$.
\vskip.2cm

By the commutativity of the hypergroup every nonzero finite dimensional variety contains an exponential. An exponential polynomial is called an {\it $m$-exponential monomial} if its variety contains only the exponential $m$. Clearly, $m$ is an $m$-exponential monomial. We define the {\it degree} of exponential monomials as follows. Exponential monomials having one dimensional variety have degree $0$, and the degree of an exponential monomial $\varphi$ is $n\geq 1$, if the degree of the exponential monomial $x\mapsto \varphi(x*y)-m(y)\varphi(x)$ is $n-1$. For instance, nonzero $m$-sine functions have degree $1$.
\vskip.2cm

For any nonnegative integer $N$ the continuous function $\varphi\colon X \to \C$  is called a \textit{generalized moment function of order $N$}, if there exist complex valued continuous functions $\varphi_k\colon X \to \C$ such that $\varphi_N=\varphi$ and 
\begin{equation*}
\varphi_k(x*y)=\sum_{j=0}^k {k \choose j}\varphi_j(x)\varphi_{k-j}(y)
\label{eq:Moment}
\end{equation*}
holds for all $k=0,1,\dots, N$ and for all $x,y$ in $X$. We say that the functions $(\varphi_k)_{k\in\left\{0,1,\dots,N \right\}}$
form a \textit{generalized moment function sequence of order $N$}. For the sake of simplicity, in this paper we shall omit the adjective "generalized" and we refer to moment functions and moment function sequences. We note that in \cite{FecGseSze20}, a more general concept of moment function sequences was introduced. 
\vskip.2cm

Observe that $\varphi_0$ is an exponential on the hypergroup $X$. In this case we say that $\varphi_0$ \textit{generates the given moment function sequence of order $N$}, and that the moment functions in this sequence {\it correspond to $\varphi_0$}. Clearly, a moment function of order $1$ corresponding to the exponential $m$ is an $m$-sine function. Given the exponential $m$, all $m$-sine functions form a linear space.
\vskip.2cm

Important examples for exponential monomials are provided by the  moment functions. Clearly, every moment function corresponding to the exponential $m$ is an $m$-exponential monomial. In particular, if the order of a generalized moment function is $N$, then it is an exponential monomial of degree at most $N$.
\vskip.2cm

Exponential monomials are the basic building blocks of spectral synthesis. We say that a variety is {\it synthesizable} if all exponential monomials in the variety span a dense subspace. We say that {\it spectral synthesis holds for} a variety if every subvariety of it is synthesisable. If every variety on $X$ is  synthesisable, then we say that {\it spectral synthesis holds on $X$}. Clearly, on every commutative hypergroup, spectral synthesis holds for finite dimensional varieties. 

\section{The main result}

The above notions suggest that generalized moment functions may play a fundamental role in the theory of spectral analysis and spectral synthesis on commutative hypergroups. In our former paper \cite{FecGseSze20}, we described generalized moment functions on commutative groups using Bell polynomials, even in the higher rank case. In fact, the notion of exponential monomials is not easy to handle, compared to that of generalized moment functions: the functional equations characterizing generalized moment functions are more convenient than those for exponential monomials. Therefore it might be fruitful to know in which situations can exponential monomials be expressed in terms of generalized moment functions. In this work we initiate the study of this problem on commutative hypergroups. The statement below is the first step towards this area.

\begin{thm}\label{main}
Let $X$ be a commutative hypergroup with identity $o$. Let \hbox{$m:X\to\C$} be an exponential, and $\varphi:X\to\C$ an $m$-exponential monomial. If the linear subspace of the variety $\tau(\varphi)$ of $\varphi$ consisting of all $m$-sine functions is one dimensional, then $\tau(\varphi)$ is generated by generalized moment functions.
\end{thm}

\begin{proof}
Suppose that $\varphi_0,\varphi_1,\dots,\varphi_n$ is a basis of $\tau(\varphi)$, and $\varphi_0=m$, $\varphi_1=s$ is an $m$-sine function, $\varphi_n=\varphi$, and the degrees of these basis functions are increasing with respect to their subscripts. In other words, we suppose that the mapping $k\mapsto \mathrm{deg}\varphi_{k}$ is increasing.
 \vskip.2cm

Then we can write for $k=1,2,\dots,n+1$:
\begin{equation}
\varphi_{n+1-k}(x*y)=\sum_{i=1}^{n+1} c_{k,i}(y)\varphi_{n+1-i}(x)
\end{equation}
for each $x,y$ in $X$. As the function $k\mapsto \deg \varphi_k$ is increasing, it follows that the matrix $c_{k,i}(y)$ is upper triangular for each $x$ (i.e. $c_{k,i}=0$ for $k>i$), and $c_{i,i}=m$ for $i=1,2,\dots,n+1$. Further
$$
\varphi_{n+1-k}\bigl((x*y)*z\bigr)=\sum_{i=1}^{n+1} c_{k,i}(z)\varphi_{n+1-i}(x*y)=
$$
$$
\sum_{i=1}^{n+1} \sum_{j=1}^{n+1} c_{k,i}(z)c_{i,j}(y)\varphi_{n+1-j}(x),
$$
and
$$
\varphi_{n+1-k}\bigl(x*(y*z)\bigr)=\sum_{j=1}^{n+1} c_{k,j}(y*z)\varphi_{n+1-j}(x)=\sum_{j=1}^{n+1} c_{k,j}(z*y)\varphi_{n+1-j}(x).
$$
Hence, by associativity 
$$
c_{k,j}(z*y)=\sum_{i=1}^{n+1}c_{k,i}(z)c_{i,j}(y)
$$
for each $y,z$ in $X$. If $C:X\mapsto \C^{(n+1)(n+1)}$ is the matrix function defined by $C(x)=\bigl(c_{i,j}(x)\bigr)$, then we have
$$
C(x*y)=C(x)C(y),\enskip\text{and}\enskip C(o)=I,
$$
where $I$ is the $(n+1)\times (n+1)$ identity matrix. By 
$$
c_{i,i+1}(x*y)=m(x)c_{i,i+1}(y)+c_{i,i+1}(x)m(y),
$$
it follows that $c_{i,i+1}$ is an $m$-sine function for each $i=1,2,\dots,n$.
\vskip.2cm

We prove the statement by induction on the dimension $n$ of $\tau(\varphi)$. First we consider the cases  $n=1,2,3,4$ separately and then we prove the statement by induction on $n\geq 4$. 
\vskip.2cm

For $n=1$ the statement is trivial, since a one dimensional variety consists of the constant multiples of an exponential, which is a generalized moment function.
\vskip.2cm

For $n=2$ the statement is obvious, because a two dimensional variety consists of the linear combinations of an exponential $m$ and an $m$-sine function, which are generalized moment functions.
\vskip.2cm

For $n=3$ we have the system of equations
\begin{eqnarray*}
\varphi_2(x*y)&=&c_{1,1}(y)\varphi_2(x)+c_{1,2}(y)\varphi_1(x)+c_{1,3}(y)\varphi_0(x)\\
\varphi_1(x*y)&=&c_{2,1}(y)\varphi_2(x)+c_{2,2}(y)\varphi_1(x)+c_{2,3}(y)\varphi_0(x)\\
\varphi_0(x*y)&=&c_{3,1}(y)\varphi_2(x)+c_{3,2}(y)\varphi_1(x)+c_{3,3}(y)\varphi_0(x),
\end{eqnarray*}
where the $c$'s are continuous complex valued functions in $\tau(\varphi)$.  Clearly, $c_{2,1}, c_{3,1}, c_{3,2}$ are zero, and we have $c_{1,1}=c_{2,2}=c_{3,3}=m$. Hence the above system can be written as
\begin{eqnarray*}
 	\varphi_2(x*y)&=&c_{1,1}(y)\varphi_2(x)+c_{1,2}(y)\varphi_1(x)+c_{1,3}(y)\varphi_0(x)\\
 	\varphi_1(x*y)&=&c_{2,2}(y)\varphi_1(x)+c_{2,3}(y)\varphi_0(x)\\
 	\varphi_0(x*y)&=&c_{3,3}(y)\varphi_0(x).
 \end{eqnarray*}
 If $C(x)=\bigl(c_{i,j}(x)\bigr)$, then we have $C(x*y)=C(x)C(y)$, and we can write
 $$
 C=
 \begin{pmatrix}
m&c_{1,2}&c_{1,3}\\
0&m&c_{2,3}\\
0&0&m
 \end{pmatrix}.
$$
 It follows that $c_{1,2}, c_{2,3}$ are $m$-sine functions, hence $c_{1,2}=\alpha_{1,2}s$, $c_{2,3}=\alpha_{2,3}s$. Then $C(x)$ has the following form:
  $$
 C=
 \begin{pmatrix}
 	m&\alpha_{1,2}s&c_{1,3}\\
 	0&m&\alpha_{2,3}s\\
 	0&0&m
 \end{pmatrix}.
 $$
As the first row of $C$ generates $\tau(\varphi)$, hence $m, c_{1,2}, c_{1,3}$ are linearly independent. It follows $\alpha_{1,2}\ne 0$. By the equation for $\varphi_1(x*y)$ above, it follows $c_{2,3}=s$, hence $\alpha_{2,3}=1\ne 0$. We have
 $$
 c_{1,3}(x*y)=m(x)c_{1,3}(y)+\alpha_{1,2} \alpha_{2,3} s(x)s(y)+c_{1,3}(x)m(y),
 $$
 and we conclude that $c_{1,1}, \frac{1}{\alpha_{1,2}}c_{1,2}, \frac{2}{\alpha_{1,2} \alpha_{2,3}}c_{1,3}$ form a generalized moment function sequence. This proves our statement for $n=3$.
 \vskip.2cm
 
Now we prove the statement for $n=4$. In that case the above notation will be modified as 
\begin{eqnarray*}
	\varphi_3(x*y)&=&c_{1,1}(y)\varphi_3(x)+c_{1,2}(y)\varphi_2(x)+c_{1,3}(y)\varphi_1(x)+c_{1,4}(x)\varphi_0(x)\\
	\varphi_2(x*y)&=&c_{2,2}(y)\varphi_2(x)+c_{2,3}(y)\varphi_1(x)+c_{2,4}(y)\varphi_0(x)\\
	\varphi_1(x*y)&=&c_{3,3}(y)\varphi_1(x)+c_{3,4}(y)\varphi_0(x)\\
	\varphi_0(x*y)&=&c_{4,4}(y)\varphi_0(x),
\end{eqnarray*}
and
$$
C=
\begin{pmatrix}
	m&\alpha_{1,2}s&c_{1,3}&c_{1,4}\\
	0&m&\alpha_{2,3}s&c_{2,4}\\
	0&0&m&\alpha_{3,4}s\\
	0&0&0&m
\end{pmatrix},
$$
where the $c$'s are continuous complex valued functions in $\tau(\varphi)$. Here again, the functions in the first row generate $\tau(\varphi)$, hence they are linearly independent. Consequently, $\alpha_{1,2}\ne 0$. On the other hand,
$$
c_{1,3}(x*y)=m(x)c_{1,3}(y)+\alpha_{1,2} \alpha_{2,3} s(x)s(y)+c_{1,3}(x)m(y),
$$
hence $\alpha_{2,3}\ne 0$: otherwise $c_{1,3}$ is an $m$-sine function, a constant multiple of $s$, which contradicts the linear independence of the functions in the first row. Finally, the equation for $\varphi_1(x*y)$ gives that $\alpha_{3,4}\ne 0$. We conclude that the functions
$$
c_{1,1}, \frac{1!}{\alpha_{1,2}}c_{1,2}, \frac{2!}{\alpha_{1,2}\alpha_{2,3}}c_{1,3}, \frac{3!}{\alpha_{1,2}\alpha_{2,3}\alpha_{3,4}}c_{1,4}
$$
form a generalized moment function sequence, which proves our statement for $n=4$.
\vskip.2cm

Suppose that it has been proved if the dimension is not greater than $n\geq 4$, and now we prove it for dimension $n+1$. Our previous notation in this general situation takes the form

\begin{eqnarray*}
	\varphi_n(x*y)&=&c_{1,1}(y)\varphi_n(x)+c_{1,n}(y)\varphi_1(x)+
	\cdots+c_{1,n+1}(y)\varphi_0(x)\\
	&\vdots&\\
	\varphi_2(x*y)&=&c_{n-1,n-1}(y)\varphi_2(x)+c_{n-1,n}(y)\varphi_1(x)+c_{n-1,n+1}(y)\varphi_0(x)\\
	\varphi_1(x*y)&=&c_{n,n}(y)\varphi_1(x)+c_{n,n+1}(y)\varphi_0(x)\\
	\varphi_0(x*y)&=&c_{n+1,n+1}(y)\varphi_0(x),
\end{eqnarray*}
and
$$
C=
\begin{pmatrix}
	m&\alpha_{1,2}s&c_{1,3}&...&c_{1,n}&c_{1,n+1}\\
	0&m&\alpha_{2,3}s&...&c_{2,n}&c_{2,n+1}\\
	\vdots& \vdots &m&...&...&...\\
	0&0&0&\ddots&c_{n-2,n}&c_{n-2,n+1}\\
	0&0&0&...&m&\alpha_{n,n+1}s\\
	0&0&0&...&0&m
\end{pmatrix}.
$$
From the fact that the functions in the first row generate $\tau(\varphi)$ we infer that they are linearly independent. The case of dimension $n$ can be applied for the variety spanned by $\varphi_0,\varphi_1,\dots,\varphi_{n-1}$ to deduce that $\alpha_{1,2}, \alpha_{2,3},\dots,\alpha_{n-1,n}$ are different from zero. Finally, the equation for $\varphi_1(x*y)$ above shows that
$$
\varphi_1(x*y)=c_{n,n}(y)\varphi_1(x)+c_{n,n+1}(y)\varphi_0(x),
$$
that is
$$
s(x*y)=m(y)s(x)+\alpha_{n,n+1}s(y)m(x),
$$
which implies $\alpha_{n,n+1}=1\ne 0$. Consequently, all the $\alpha$'s are nonzero. Then we let $f_0=m$ and  for $k=1,2,\dots,n$
$$
f_k=\frac{k!}{\alpha_{1,2}\alpha_{2,3}\cdots \alpha_{k,k+1}}c_{1,k+1}.
$$ 
We show that $f_0,f_1,\dots,f_n$ form a generalized moment function sequence of order $X$. We have
{\small 
$$
f_k(x*y)=\frac{k!}{\alpha_{1,2}\cdots \alpha_{k,k+1}}c_{1,k+1}(x*y)
=\frac{k!}{\alpha_{1,2}\cdots \alpha_{k,k+1}} \sum_{j=1}^{k+1} c_{1,j}(x) c_{j,k+1}(y)=
$$
}
{\tiny
$$
\frac{k!}{\alpha_{1,2}\cdots \alpha_{k,k+1}} \sum_{j=0}^{k} \frac{\alpha_{1,2}\cdots \alpha_{j,j+1}}{j!}f_j(x) c_{j+1,k+1}(y)
=\sum_{j=0}^{k} \frac{k!}{j!}\frac{1}{\alpha_{j+1,j+2}\cdots \alpha_{k,k+1}}f_j(x) c_{j+1,k+1}(y)=
$$
}
{\tiny
$$
\sum_{j=0}^{k} \frac{k!}{j!}\frac{1}{\alpha_{j+1,j+2}\cdots \alpha_{k-1,k}}f_j(x) \frac{c_{j,k}(y)}{\alpha_{j,j+1}}
=\sum_{j=0}^{k} \frac{k!}{j!}\frac{1}{\alpha_{j,j+1}\cdots \alpha_{k-1,k}}f_j(x) c_{j,k}(y)=
$$
$$
\sum_{j=0}^{k} \frac{k!}{j!}\frac{1}{\alpha_{j,j+1}\cdots \alpha_{k-2,k-1}}f_j(x) \frac{c_{j-1,k-1}(y)}{\alpha_{j-1,j}}
=\sum_{j=0}^{k} \frac{k!}{j!}\frac{1}{\alpha_{j-1,j}\cdots \alpha_{k-2,k-1}}f_j(x) c_{j-1,k-1}(y).
$$
}

Continuing this process we arrive at
{\small
$$
f_k(x*y)=\sum_{j=0}^{k} \frac{k!}{j!}\frac{1}{\alpha_{2,3}\alpha_{3,4}\cdots \alpha_{k-j,k-j+1}}f_j(x)\frac{c_{1,k-j+1}(y)}{\alpha_{1,2}}=
$$
$$
\sum_{j=0}^{k} \frac{k!}{j!(k-j)!}\frac{(k-j)!}{\alpha_{1,2}\alpha_{2,3}\cdots \alpha_{k-j,k-j+1}}f_j(x)c_{1,k-j+1}(y)=
\sum_{j=0}^{k} \binom{k}{j}f_j(x)f_{k-j}(y),
$$
}
which proves the statement.
\end{proof}

One may ask how restrictive is the condition that the sine functions in a variety form a one dimensional linear space. Of course, this requirement is quite restrictive, but still there are large, important classes of commutative hypergroups having this property. This condition means a kind of "one-dimensionality" of the hypergroup. These classes include all polynomial hypergroups in one variable, Sturm--Liouville hypergroups, etc., where, in fact, every finite dimensional variety has this property.  For instance, on polynomial hypergroups we can apply our result as follows.

\begin{cor}
Let $X$ be a polynomial hypergroup associated with the sequence of polynomials $\bigl(P_n\bigr)_{n\in\N}$. Then every complex valued function on $\N$ (i.e. every complex sequence) is the pointwise limit of linear combinations of generalized moment functions on $X$. 
\end{cor}

\begin{proof}
By Theorem 6.7. in \cite{Sze12}, spectral synthesis holds on every polynomial hypergroup. This means that, in every variety the exponential polynomials span a dense subspace. If $f:\N\to\C$ is any function, then we can apply this result for the variety $\tau(f)$ of $f$. Consequently, to prove our statement it is enough to show that for each exponential $m$ on $X$, all $m$-sine functions form a one dimensional linear space. Let $m$ be an exponential on $X$. By Theorem 2.2. in \cite{Sze12}, there exists a complex number $\lambda$ such that $m(n)=P_n(\lambda)$ holds for each $n$ in $\N$. On the other hand, by Theorem 2.5. in \cite{Sze12}, every $m$-sine function $s$ on $X$ has the form $s(n)=c P'_n(\lambda)$ with some complex number $c$. It follows that all $m$-sine functions form a one dimensional linear space, hence by Theorem \ref{main}, our statement follows.
\end{proof}

The following result can be obtained on Sturm--Liouville hypergroups. Here $\R_0$ denotes the set of nonnegative reals.

\begin{cor}
Let $X=(\R_0,A)$ be the Sturm--Liouville hypergroup associated with the Sturm--Liouville function $A:\R_0\to \R$. Let $V$ be a synthesizable variety on $X$. Then every function in $V$ is the uniform limit on compact sets of linear combinations of generalized moment functions on $X$. 
\end{cor}

\begin{proof}
Applying a similar argument to that in the previous theorem it is enough to show that for each exponential on $X$, the linear space of all $m$-sine functions in the variety of an arbitrary $m$-exponential monomial is one-dimensional.
\vskip.2cm

By Theorem 4.2. in \cite{Sze12}, the function $m:\R_0\to \C$ is an exponential if and only if it is twice continuously differentiable and there exists a complex number $\lambda$ such that 
\begin{equation}\label{exp}
m''(x)+\frac{A'(x)}{A(x)}m'(x)=\lambda m(x)
\end{equation}
holds for $x>0$, further $m(0)=1$, $m'(0)=0$. Suppose that $m$ satisfies \eqref{exp}. Then, by Theorem 4.5. in \cite{Sze12}, the function $s:\R_0\to \C$ is an $m$-sine function if and only if it is twice continuously differentiable and there exists a complex number $c$ such that 
\begin{equation}\label{sin}
	s''(x)+\frac{A'(x)}{A(x)}s'(x)-\lambda s(x)=c m(x)
\end{equation}
holds for $x>0$, further $s(0)=0$, $s'(0)=0$. Let $s_0$ be the unique twice continuously differentiable function satisfying
\begin{equation}\label{sin0}
	s_0''(x)+\frac{A'(x)}{A(x)}s_0'(x)-\lambda s_0(x)=m(x)
\end{equation}
for $x>0$, and $s_0(0)=0$, $s_0'(0)=0$. It is known that this problem has a unique solution, hence $s_0$ is unique. On the other hand, if $s$ is any $m$-sine function, then there is a unique $c$ such that $s$ satisfies the problem \eqref{sin}. However, $c s_0$ also satisfies \eqref{sin}, hence we infer $s=c s_0$, and the proof is complete.

\end{proof}

\vskip10cm

\end{document}